\documentclass[reqno,12pt]{amsart}

\usepackage{tikz}
\usepackage{amsmath}
\usepackage{amsfonts}
\usepackage{amssymb,euscript, mathrsfs,amscd} 
\usepackage{epsfig}
\usepackage{a4wide}
\usepackage{blkarray}
\usepackage{multirow}
\usepackage{mathtools}
\usepackage{bm}
\usepackage{ytableau}
\usepackage{caption}
\usepackage{hyperref}
\usepackage{cases}
\usepackage[shortlabels]{enumitem}

\usepackage{graphicx,subfigure}

\usetikzlibrary{decorations.markings}

\newtheorem{thm}{Theorem}[section]
\newtheorem{prop}[thm]{Proposition}
\newtheorem{lem}[thm]{Lemma}

\theoremstyle{definition}

\newtheorem{exam}[thm]{Example}
\newtheorem{rmk}[thm]{Remark}

\theoremstyle{remark}

\newcommand\inv{\operatorname{inv}}

\newcommand\sgn{\operatorname{sgn}}

\newcommand\col{\operatorname{col}}

\newcommand\Des{\operatorname{Des}}

\newcommand\SSYT{\operatorname{SSYT}}
\newcommand\Par{\operatorname{Par}}
\newcommand\plac{\operatorname{plac}}

\newcommand\QQ{\mathbb{Q}}
\newcommand\NN{\mathbb{N}}
\newcommand\ZZ{\mathbb{Z}}

\newcommand\Sym{\mathsf{Sym}}

\newcommand\UU{\mathcal{U}}

\newcommand\II{\mathcal{I}}

\newcommand\calN{\mathcal{N}}

\newcommand\xx{\mathbf{x}}
\newcommand\yy{\mathbf{y}}
\newcommand\uu{\mathbf{u}}

\newcommand\ee{\mathfrak{e}}
\newcommand\hh{\mathfrak{h}}
\newcommand\sch{\mathfrak{J}}
\newcommand\pp{\mathfrak{p}}
\newcommand\mono{\mathfrak{m}}
\newcommand\ff{\mathfrak{f}}

\newcommand\SYM{\mathfrak{S}}

\newcommand\ww{\mathsf{w}}
\newcommand\vv{\mathsf{v}}

\newcommand\equivclass[1]{#1/{\sim}}
\newcommand\qand{\quad\mbox{and}\quad}

\title{Noncommutative symmetric functions and skewing operators}

\author{Byung-Hak Hwang}
\email{byunghakhwang@gmail.com}


\begin{document}
\begin{abstract}
  Skewing operators play a central role in the symmetric function theory
  because of the importance of the product structure of the symmetric function
  space.
  The theory of noncommutative symmetric functions is a useful tool for studying
  expansions of a given symmetric function in terms of various bases.
  In this paper, we establish a further development of the theory
  for studying skewing operators.
  Using this machinery, we are able to easily reproduce the Littlewood--Richardson
  rule, and provide recurrence relations for chromatic quasisymmetric functions,
  which generalizes Harada--Precup's recurrence.
\end{abstract}

\maketitle

\section{Introduction} \label{sec:intro}
The theory of symmetric functions plays a significant role in algebra combinatorics.
Although the definition of symmetric functions is rooted in combinatorics,
symmetric functions appear in, beyond combinatorics, many other fields, e.g.,
representation theory, algebraic geometry, probability theory, and more.
Symmetric functions are formal power series on infinite variables that are invariant
under permuting the variables. Therefore, a linear sum of symmetric functions,
and a product of symmetric functions are also symmetric functions.
In other words, they carry an algebra structure.

Whenever a novel class of symmetric functions is introduced,
a natural question for the linear structure arises:
how can we expand the symmetric functions in terms of several symmetric function
bases?
The reason why this question is crucial is that coefficients in such expansions hold
the potential to have a meaning in diverse area,
e.g., the dimensions of specific modules, multiplicities of irreducible
representations, or intersection numbers of irreducible varieties.

Furthermore, the significance of the space of symmetric functions extends beyond its
linear structure. Since the space is isomorphic to other spaces appearing in various
fields as an algebra,
the product on the symmetric function space has a meaningful interpretation
on opposite sides, so the product structure becomes equally important.
Thus the multiplicative property of symmetric functions have been studied
in the literature.
When an algebra is equipped with an inner product, a efficient way to study
the multiplication structure is to consider the adjoint of the
multiplication with respect to the inner product.
In the symmetric function space, for a symmetric function \( f(\xx) \),
the skewing operator \( f^\perp \) is defined as the adjoint of
the multiplication by \( f(\xx) \).
(For an explicit definition, see Section~\ref{sec:skewing}.)
Hence, these skewing operators play an important role, and have been studied
extensively.

Numerous tools have been developed to expand a given symmetric function into several
bases, and find a combinatorial interpretation for coefficient in the expansions.
One such tool that we will explore in this paper is the theory of
noncommutative symmetric functions. 
In their seminal paper~\cite{FG98}, Fomin and Greene introduced noncommutative Schur
functions for studying the Schur expansion of a given symmetric function.
Using this machinery, they reproduced the Littlewood--Richardson rule,
and established the Schur positivity of Stanley symmetric functions.
Building upon this foundation, Blasiak and Fomin~\cite{BF17} developed a more
general algebraic framework for the noncommutative symmetric function theory.

In the current work, we present a further development of the noncommutative symmetric
function theory.
More explicitly, we take one more step from the foundation in \cite{BF17}
to apply the theory to studying skewing operators.
We provide a way to calculate a given symmetric function `skewed' by certain
symmetric functions via the theory of noncommutative symmetric functions
(Theorem~\ref{thm:skew_Omega}).

The remainder of the paper is organized as follows.
In Section~\ref{sec:skewing}, we recast some basic of the symmetric function theory,
and define skew operators. We also present a variant~\eqref{eq:cauchy_product_skew}
of the Cauchy product formula.
Section~\ref{sec:noncomm} gives a brief review of the noncommutative symmetric
function theory (we follow notations from \cite{BF17}), and developes the theory
to make applicable to skewing operators.
In Section~\ref{sec:application}, we apply the machinery developed in the previous
section to complete homogeneous symmetric functions, Schur functions, and
chromatic quasisymmetric functions. For the last case, we need a generalization of
noncommutative symmetric functions, introduced in \cite{Hwa22}.
As results, we reproduce the Littlewood--Richardson rule, and provide two
recurrence relations (Theorems~\ref{thm:recur_e_perp} and \ref{thm:recur_p_perp})
for the chromatic quasisymmetric functions.
One of the relations generalizes the Harada--Precup recurrence~\cite{HP19}.

\section{Skewing operators} \label{sec:skewing}
In this section, we present some well-known facts in the theory of symmetric
functions, which will be used in next sections. In addition, we give a variant
of the Cauchy product formula concerning skewing operators.

We assume that the reader is familiar with the basics of the theory of symmetric
and quasisymmetric functions. Therefore, we will only present some well-known
facts that we will use, without proofs. Their proofs can be found in
\cite[Chapter~7]{Sta99}.

Let $\xx=(x_1, x_2, \dots)$ be a set of formal variables that commute with each
other.
Let \( \Sym_n \) be the \( \QQ \)-space of symmetric functions of homogeneous
degree \( n \), and \( \Sym = \bigoplus_{n\ge 0} \Sym_n \).
We denote by \( \Par \) the set of integer partitions.
For \( k\ge 1 \), the \emph{elementary symmetric function} \( e_k(\xx) \)
is defined as
\[
  e_k(\xx) = \sum_{i_1>\dots > i_k} x_{i_1}\cdots x_{i_k},
\]
and for a partition \( \lambda=(\lambda_1,\dots,\lambda_\ell) \),
define \( e_\lambda(\xx)=e_{\lambda_1}(\xx)\cdots e_{\lambda_\ell}(\xx) \).
The following is a crucial fact in the theory of symmetric functions.
\begin{thm}[{\cite[Theorem~7.4.4]{Sta99}}]
  The elementary symmetric functions \( e_1(\xx), e_2(\xx), \dots \) are
  algebraically independent, and thus \( \{e_\lambda(\xx)\}_{\lambda\in\Par} \)
  forms a basis for \( \Sym \).
  Equivalently, the \( \QQ \)-algebra \( \Sym \) is a free associative commutative
  \( \QQ \)-algebra over \( e_1(\xx), e_2(\xx), \dots \).
\end{thm}

We also present other bases and provide some relations with elementary
symmetric functions:
\begin{enumerate}
  \item The \emph{monomial symmetric function} \( m_\lambda(\xx) \) is defined
  to be
  \[
    m_\lambda(\xx) = \sum_{\alpha} x^\alpha,
  \]
  where $\alpha$ ranges over all distinct permutations $\alpha=(\alpha_1, \alpha_2,
  \dots)$ of the entries of $\lambda$ and $x^\alpha = x_1^{\alpha_1} x_2^{\alpha_2}
  \cdots$.
  \item The \emph{complete homogeneous symmetric function} \( h_k(\xx) \) is
  defined to be
  \[
    h_k(\xx) = \sum_{i_1\le \dots \le i_k} x_{i_1}\cdots x_{i_k},
  \]
  and \( h_\lambda(\xx)=h_{\lambda_1}(\xx)\cdots h_{\lambda_\ell}(\xx) \)
  for a partition \( \lambda \). They satisfy the following relation:
  \begin{equation} \label{eq:relation_he}
    h_k(\xx) - e_1(\xx) h_{k-1}(\xx) + \dots + (-1)^k e_k(\xx) = \delta_{k,0}.
  \end{equation}
  \item The \emph{power sum symmetric function} \( p_k(\xx) \) is defined to be
  \[
    p_k(\xx) = \sum_{i} x_i^k,
  \]
  and \( p_\lambda(\xx)=p_{\lambda_1}(\xx)\cdots p_{\lambda_\ell}(\xx) \)
  for a partition \( \lambda \). We also have the following identity:
  \begin{equation} \label{eq:relation_pe}
    p_k(\xx) = e_1(\xx) h_{k-1}(\xx) - 2 e_2(\xx) h_{k-2}(\xx) + \cdots + (-1)^{k-1} k e_k(\xx).
  \end{equation}
  \item The \emph{Schur function} \( s_\lambda(\xx) \) is defined by
  the (dual) Jacobi--Trudi identity: for a partition \( \lambda \),
  \begin{align}
    s_\lambda(\xx)
      &= \det (e_{\lambda'_i+j-i}(\xx))_{i,j=1}^{m} \nonumber \\
      &= \sum_{\sigma\in\SYM_m} \sgn(\sigma)
      e_{\lambda'_1-\sigma_1+1}(\xx) e_{\lambda'_2-\sigma_2+2}(\xx)\cdots
      e_{\lambda'_m-\sigma_m+m}(\xx),  \label{eq:relation_se}
  \end{align}
  where \( m=\lambda_1 \) and \( \lambda' \) is the conjugate of \( \lambda \).
\end{enumerate}

Let $\yy=(y_1,y_2,\dots)$ be another set of commuting variables. Define
\[
  C(\xx,\yy) = \prod_{i,j} \frac{1}{1-x_iy_j} \in\mathbb{Q}[[\xx,\yy]],
\]
called the \emph{Cauchy product}.
Observing the definition carefully, one can deduce
\[
  C(\xx,\yy) = \sum_{\lambda\in\Par} m_\lambda(\xx) h_\lambda(\yy).
\]
The \emph{Hall inner product} \( \langle - , - \rangle \) is a symmetric bilinear
pairing on \( \Sym\times\Sym \) such that the monomial symmetric functions and
the complete homogenous symmetric functions are dual to each other with respect to
the pairing, that is,
\[
  \langle m_\lambda(\xx), h_\mu(\xx) \rangle = \delta_{\lambda\mu}.
\]
Furthermore, we have a deeper relationship between the Cauchy product and the Hall
inner product.
\begin{prop}[{\cite[Lemma~7.9.2]{Sta99}}]
  Let \( \{f_\lambda(\xx)\}_{\lambda\in\Par} \) and \( \{g_\lambda(\xx)\}
  _{\lambda\in\Par} \) be two bases for \( \Sym \). Then
  \[
    C(\xx,\yy) = \sum_{\lambda\in\Par} f_\lambda(\xx) g_\lambda(\yy)
  \]
  if and only if
  \[
    \langle f_\lambda(\xx), g_\mu(\xx) \rangle = \delta_{\lambda\mu}
  \]
  for any \( \lambda,\mu\in\Par \).
\end{prop}

We are now in a position to define a skewing operator.
Let \( a(\xx) \) be a symmetric function.
The \emph{skewing operator} \( a^\perp \) is a linear map on \( \Sym \)
\[
  a^\perp : f(\xx) \longmapsto a^\perp f(\xx),
\]
where \( a^\perp f(\xx) \) is a unique symmetric function satisfying
\[
  \langle a^\perp f(\xx), g(\xx) \rangle = \langle f(\xx), a(\xx) g(\xx) \rangle
\]
for any \( g(\xx)\in\Sym \).
The terminology `skewing' may have originated from the following concept.
Let \( \lambda \) and \( \mu \) be two partitions with \( \lambda_i \ge \mu_i \)
for all \( i \). Then the \emph{skew Schur function} \( s_{\lambda/\mu}(\xx) \)
is defined to be
\[
  s_{\lambda/\mu}(\xx) = \det (e_{\lambda'_i-\mu'_j-i+j})_{i,j=1}^m,
\]
where \( m=\lambda_1 \). It is highly non-trivial that
\begin{equation} \label{eq:skewing_s=s_skew}
  s_\mu^\perp s_\lambda(\xx) = s_{\lambda/\mu}(\xx).
\end{equation}

Furthermore, let us consider two symmetric function bases
\( \{f_\lambda(\xx)\}_{\lambda\in\Par} \) and
\( \{g_\lambda(\xx)\}_{\lambda\in\Par} \), where they are dual to each other.
Equivalently,
\[
  C(\xx,\yy) = \sum_{\lambda\in\Par} f_\lambda(\xx) g_\lambda(\yy).
\]
Let \( d^\nu_{\lambda\mu} \) stand the structure constant for
\( \{f_\lambda(\xx)\}_{\lambda\in\Par} \), that is,
\[
  f_\lambda(\xx) f_\mu(\xx) = \sum_\nu d^\nu_{\lambda\mu} f_\nu(\xx).
\]
Then we obtain a simple identity about skewing \( g_\lambda(\xx) \)
by \( f_\mu(\xx) \), and have a variant of the Cauchy product formula.
\begin{prop}
  For partitions \( \lambda \) and \( \mu \), we have
  \begin{equation} \label{eq:a_perp=sum_db}
    f_\mu^\perp g_\lambda(\xx) = \sum_{\nu\in\Par} d^\lambda_{\mu\nu} g_\nu(\xx),
  \end{equation}
  and
  \begin{equation} \label{eq:cauchy_product_skew}
    f_\mu(\xx) C(\xx,\yy) = \sum_{\lambda\in\Par} f_\lambda(\xx) f_\mu^\perp g_\lambda(\yy).
  \end{equation}
\end{prop}
\begin{proof}
  The first equation follows immediately from the duality of the two bases, and
  the second one also can be easily obtained as follows:
  \begin{align*}
    f_\mu(\xx) C(\xx,\yy)
      &= \sum_{\nu\in\Par} f_\mu(\xx) f_\nu(\xx) g_\nu(\yy) \\
      &= \sum_{\lambda\in\Par} f_\lambda(\xx) \sum_{\nu\in\Par} d^\lambda_{\nu\mu} g_\nu(\yy) \\
      &= \sum_{\lambda\in\Par} f_\lambda(\xx) f_\mu^\perp g_{\lambda}(\yy)
  \end{align*}
  by \eqref{eq:a_perp=sum_db}.
\end{proof}

\section{Noncommutative symmetric functions} \label{sec:noncomm}
In this section, we begin by providing an overview of the theory of noncommutative
symmetric functions, and then we present a further development using the variant
of the Cauchy product~\eqref{eq:cauchy_product_skew}.

The theory of noncommutative symmetric functions is a powerful tool for studying the
linear structure of symmetric functions. The theory gives a framework for expanding a
given symmetric function in terms of a specific symmetric function basis, and
especially for showing the positivity of this basis.
The theory was initially introduced by Fomin and Greene in \cite{FG98}.
Later, in \cite{BF17}, Blasiak and Fomin presented a more comprehensive algebraic
framework for noncommutative symmetric functions.
We follow notations in \cite{BF17}.
(Moreover, in \cite{Hwa22,BEPS22}, the authors extended the theory to apply to
posets. We will revisit their generalization in Section~\ref{sec:chrom_sym}.)

Fix an integer \( N\ge 1 \). Let \( \UU=\QQ \langle u_1,\dots,u_N \rangle \) be
the free associative \( \QQ \)-algebra generated by \( \{u_1,\dots, u_N\} \).
We will regard the generators \( u_1,\dots,u_N \) as noncommuting variables.
To simplify notation, we write \( u_\ww = u_{\ww_1}\cdots u_{\ww_d} \) for a word
\( \ww=\ww_1\cdots\ww_d \) on \( [N] \).
For \( k\ge 1 \), the \emph{noncommutative elementary symmetric function}
\( \ee_k(\uu) \) is defined by
\[
  \ee_k(\uu) = \sum_{N\ge i_1>\dots>i_k\ge 1} u_{i_1}\cdots u_{i_k}.
\]
We also define an analogue of the complete homogeneous symmetric function using
the relation~\eqref{eq:relation_he}.
For \( k\ge 1 \), the \emph{noncommutative complete homogeneous symmetric function}
\( \hh_k(\uu) \) is defined as
\[
  \hh_k(\uu) = \ee_1(\uu) \hh_{k-1}(\uu)\ - \dots + (-1)^{k-1} \ee_k(\uu)
\]
inductively. For a partition \( \lambda=(\lambda_1,\dots,\lambda_\ell) \),
let \( \hh_\lambda(\uu) = \hh_{\lambda_1}(\uu)\cdots\hh_{\lambda_\ell}(\uu) \).
It is an easy exercise that the noncommutative complete homogeneous symmetric
function \( \hh_k(\uu) \) has the following simple expression:
\begin{equation} \label{eq:hh=increasing}
  \hh_k(\uu) = \sum_{i_1\le \cdots \le i_k} u_{i_1}\cdots u_{i_k}.
\end{equation}

For a 2-sided ideal \( \II \) of \( \UU \), we say that \emph{\( \II \) satisfies
the commutation relation} if
\[
  \ee_k(\uu) \ee_\ell(\uu) \equiv \ee_\ell(\uu) \ee_k(\uu) \mod \II
\]
for all \( k,\ell\ge 1 \).
In \cite{BF17}, a necessary and sufficient condition for which \( \II \)
satisfies the commutation relation is provided.
Let \( \phi \) be a map given by
\[
  e_k(\xx)\in\Sym \mapsto \ee_k(\uu)\in\UU.
\]
Since \( \UU \) is not commutative, the map \( \phi \) cannot be extended
to an algebra homomorphism.
However, when \( \II \) satisfies the commutation relation,
we have the following proposition, which will be used frequently.
\begin{prop} \label{prop:commutation}
  Let \( \II \) be a 2-sided ideal of \( \UU \) satisfying the commutation relation.
  Then the map
  \[
    \phi: \Sym\rightarrow \UU/\II
  \]
  can be extended to an algebra homomorphism.
  In particular, the image of every identity on \( \Sym \) under the map
  holds on \( \UU/\II \) as a congruence.
\end{prop}
\begin{proof}
  The proof follows from the commutation relation and the fact that \( \Sym \) is a
  free commutative algebra generated by \( e_1(\xx), e_2(\xx), \dots \).
\end{proof}

We now assume that the commuting variables \( \xx \) and \( x \) also commute
with \( \uu \).
Let
\[
  H(x,\uu) = \sum_{\ell\ge 0} x^\ell \hh_\ell(\uu) \in\UU[[x]].
\]
We define the \emph{noncommutative Cauchy product} to be
\[
  \Omega(\xx, \uu) = H(x_1, \uu)H(x_2, \uu)\cdots \in\UU[[\xx]].
\]
Then one can easily observe that \( \Omega(\xx,\uu) \) can be written as
\[
  \Omega(\xx,\uu) = \sum_\alpha M_\alpha(\xx) \hh_\alpha(\uu),
\]
where $\alpha=(\alpha_1,\dots,\alpha_\ell)$ ranges over all compositions,
\( M_\alpha(\xx) \) is the monomial quasisymmetric function,
and $\hh_\alpha(\uu) = \hh_{\alpha_1}(\uu)\cdots\hh_{\alpha_\ell}(\uu)$.
Moreover, by \eqref{eq:hh=increasing}, we have
\begin{equation} \label{eq:Omega=Fw}
  \Omega(\xx,\uu) = \sum_{\ww} F_{d, \Des(\ww)}(\xx) u_\ww,
\end{equation}
where the sum ranges over all words on \( [N] \), \( d \) is the length of \( \ww \),
\( F_{d,S}(\xx) \) is the fundamental quasisymmetric function, and
\( \Des(\ww) = \{i\in [d-1] \mid \ww_i > \ww_{i+1}  \} \).

On the other hand, let \( \UU^*\) be the free \( \ZZ \)-module generated by
words on \( [N] \), and \( \langle -,- \rangle \) a pairing between \( \UU \) and
\( \UU^* \) given by \( \langle u_\ww, \vv \rangle = \delta_{\ww,\vv} \)
for words \( \ww \) and \( \vv \). In addition, we naturally extend the pair
to one between \( \UU[[\xx]] \) and \( \UU^* \).
For an ideal \( \II \), let 
\[
  \II^\perp = \{\gamma\in\UU^* \mid \mbox{\( \langle z, \gamma \rangle = 0 \) for all \( z\in\II \)} \}.
\]
If \( \gamma = \sum_\ww \gamma_\ww \ww \in \UU^* \),
then from \eqref{eq:Omega=Fw} we have
\[
  F_\gamma(\xx)
  := \sum_\ww \gamma_\ww F_{d, \Des(\ww)}(\xx)
  = \langle \Omega(\xx,\uu), \gamma \rangle.
\]
Furthermore, we have the following theorem.
\begin{thm}[{\cite[Propositions~2.6 and 2.9]{BF17}}] \label{thm:Omega=gf}
  Suppose that \( \II \) satisfies the commutation relation and
  \( \gamma\in\II^\perp \).
  Then \( F_\gamma(\xx) \) is a symmetric function.
  Moreover, if we can write
  \[
    \Omega(\xx,\uu)\equiv \sum_{\lambda\in\Par} g_\lambda(\xx) \ff_\lambda(\uu)
    \mod \II[[\xx]]
  \]
  for some symmetric function basis $g_\lambda(\xx)$ and noncommutative
  symmetric functions $\ff_\lambda(\uu)$, then we have
  \[
    F_\gamma(\xx) = \sum_{\lambda\in\Par} g_\lambda(\xx) \langle \ff_\lambda(\uu), \gamma \rangle.
  \]
\end{thm}

To expand a given symmetric function \( F(\xx) \) in terms of a basis
\( \{g_\lambda\}_{\lambda\in\Par} \),
Theorem~\ref{thm:Omega=gf} provides the following strategy.
First, find a suitable \( \gamma \) and \( \II \) such that
\( F(\xx) = F_\gamma(\xx) \) where \( \II \) satisfies the commutation relation
and \( \gamma\in\II^\perp \).
Next, define a noncommutative analogue \( \ff_\lambda(\uu) \) of the dual basis
\(\{f_\lambda(\xx)\}_{\lambda\in\Par} \) of \(\{g_\lambda(\xx)\}_{\lambda\in\Par} \).
Finally, find a monomial expression of \( \ff_\lambda(\uu) \) modulo \( \II \) and
apply Theorem~\ref{thm:Omega=gf} to obtain the desired expansion.

We now present a noncommutative version of the variant \eqref{eq:cauchy_product_skew}
of the Cauchy product.
\begin{thm} \label{thm:skew_Omega}
  Suppose that two bases \( \{f_\lambda(\xx)\}_{\lambda\in\Par} \) and
  \( \{g_\lambda(\xx)\}_{\lambda\in\Par} \) are dual to each other.
  Let \( \II \) be a 2-sided ideal satisfying the commutation relation,
  and \( \ff_\lambda(\uu) := \phi(f_\lambda(\xx)) \) where \( \phi \) is the map
  in Proposition~\ref{prop:commutation} for \( \lambda\in\Par \). Then we have
  \begin{equation} \label{eq:skew_Omega}
    \ff_\mu(\uu) \Omega(\xx,\uu)
    \equiv \sum_{\lambda\in\Par} \ff_\lambda(\uu) f_\mu^\perp g_\lambda(\xx)
    \mod \II[[\xx]].
  \end{equation}
  Consequently, let \( \gamma\in\II^\perp \), then
  \[
    f_\mu^\perp F_\gamma(\xx) = \langle \ff_\mu(\uu) \Omega(\xx,\uu), \gamma \rangle.
  \]
  In particular, the coefficient of \( g_\lambda(\xx) \) of \( f_\mu^\perp F_\gamma
  (\xx) \) is equal to \( \langle \ff_\mu(\uu) \ff_\lambda(\uu), \gamma \rangle \).
\end{thm}
\begin{proof}
  The identity~\eqref{eq:skew_Omega} follows from \eqref{eq:cauchy_product_skew}
  and Proposition~\ref{prop:commutation}. Furthermore, we have
  \begin{align*}
    f_\mu^\perp F_\gamma(\xx)
      &= \sum_\lambda f_\mu^\perp g_\lambda(\xx) \langle \ff_\lambda(\uu), \gamma \rangle && \mbox{by Theorem~\ref{thm:Omega=gf},} \\
      &= \langle \ff_\mu(\uu) \Omega(\xx,\uu), \gamma \rangle && \mbox{by \eqref{eq:skew_Omega},} 
  \end{align*}
  as desired.
\end{proof}
Therefore, if we have equipped an appropriate ideal \( \II \) and \( \gamma \)
such that \( F(\xx) = F_\gamma(\xx) \) and \( \gamma\in\II^\perp \),
then knowing \( f_\mu^\perp F(\xx) \) is essentially equivalent
to having a monomial expression of \( \ff_\mu(\uu) \) modulo \( \II \).

\section{Applications} \label{sec:application}
In this section, we use Theorem~\ref{thm:skew_Omega} to study skewing certain
symmetric functions by several bases.

\subsection{Complete homogeneous symmetric functions}
In this section, we apply Theorem~\ref{thm:skew_Omega} to complete homogeneous
symmetric functions \( h_\lambda(\xx) \).
While the results exhibited in this section can be obtained directly from the
definitions, we will nonetheless use the noncommutative symmetric
function machinery as a warm-up exercise. Moreover, the results will be utilized
in the subsequent sections.

To obtain complete homogeneous symmetric functions, we first need to define
an ideal properly. In this case, let \( \II_0 \) be the 2-sided ideal of
\( \UU \) generated by
\[
  u_a u_b - u_b u_a \quad \mbox{(\( 1\le a < b \le N \))}.
\]
Of course, \( \II_0 \) satisfies the commutation relation.
In fact, the quotient \( \UU/\II_0 \) is isomorphic
to the polynomial ring \( \QQ[y_1,\dots,y_N] \), and then noncommutative
symmetric functions to be defined in the sequel have simple monomial expressions
that are the same as those for ordinary symmetric functions.

For a partition \( \lambda \), let \( [\lambda] \) be the set of all distinct
permutations of the letters, \( \lambda_1 \) 1's, \( \lambda_2 \) 2's,
\( \dots \), \( \lambda_\ell \) \( \ell \)'s.
Here, we assume that \( N \ge \ell(\lambda) \).
For example, if \( \lambda = (2,1,1) \), then
\[
  [\lambda] = \{\mathsf{1123, 1132, 1213, 1231, 1312, 1321, 2113, 2131, 2311, 3112, 3121, 3211}  \}.
\]
Define
\[
  \gamma_\lambda = \sum_{\ww\in [\lambda]} \ww \in \UU^*,
\]
so it is easy to check that \( \gamma_\lambda \in \II_0^\perp \).

Recall the Cauchy product formula:
\[
  C(\xx,\yy) = \sum_{\lambda\in\Par} m_\lambda(\xx) h_\lambda(\yy).
\]
For a partition \( \lambda \), define \( \mono_\lambda(\uu) := \phi(m_\lambda(\xx))
\in\UU/\II_0 \). Then by Proposition~\ref{prop:commutation}, we have
\begin{equation} \label{eq:Omega=mh}
  \Omega(\xx,\uu) \equiv \sum_{\lambda\in\Par} \mono_\lambda(\uu) h_\lambda(\xx)
  \mod \II_0[[\xx]].
\end{equation}
Then, we can easily show the following proposition from Theorem~\ref{thm:Omega=gf}.
\begin{lem}
  We have
  \[
    F_{\gamma_\lambda}(\xx) = h_\lambda(\xx).
  \]
\end{lem}
\begin{proof}
  We only need to find a monomial expression for \( \mono_\lambda(\uu) \).
  But, as mentioned above, the quotient \( \UU/\II_0 \) is isomorphic to
  the polynomial ring with \( n \) variables, and thus the noncommutative analogue
  of monomial symmetric functions can be expressed as same as the ordinary monomial
  symmetric function. Then, we have
  \begin{equation} \label{eq:mono=mono}
    \mono_\lambda(\uu) \equiv
    \sum_{\alpha} \sum_{1\le i_1<\dots< i_k\le N} u_{i_1}^{\alpha_1}\cdots u_{i_k}^{\alpha_k} \mod \II_0,
  \end{equation}
  where \( \alpha \) ranges over all distinct permutations of \( \lambda \),
  and \( k=\ell(\lambda) \).
  Note that indices in each monomial appearing in \eqref{eq:mono=mono} are weakly
  increasing.
  On the other hand, there is a unique weakly increasing word in \( [\lambda] \).
  Therefore, by Theorem~\ref{thm:Omega=gf} and \eqref{eq:Omega=mh}, we deduce
  \[
    F_{\gamma_\lambda} = \sum_\mu h_\mu(\xx) \langle \mono_\mu(\uu), \gamma_\lambda \rangle = h_\lambda(\xx).
  \]
\end{proof}

Let us now skew the complete homogeneous functions.
We first consider the skewing operator \( e_k^\perp \).
For \( i\in [N] \), let \( \epsilon_i \) be the (0,1)-vector of length \( N \)
whose unique nonzero entry is the \( i \)-th entry.
For a subset \( S\subseteq [N] \), let
\[
  \epsilon_S = \sum_{i\in S} \epsilon_i.
\]
Before stating a result, we fix conventional notation.
For \( k < 0 \), let \( h_k(\xx) = 0 \) and \( h_0(\xx) = 1 \).
In addition, for an integer sequence \( \alpha=(\alpha_1,\dots,\alpha_d) \),
we let \( h_\alpha(\xx) = h_{\alpha_1}(\xx)\cdots h_{\alpha_d}(\xx) \).
\begin{prop} \label{prop:e_perp_h}
  For a partition \( \lambda \) and an integer \( k\ge 1 \),
  we have
  \[
    e_k^\perp h_\lambda(\xx)
    = \sum_{S} h_{\lambda-\epsilon_S}(\xx),
  \]
  where \( S \) ranges over all \( k \)-subsets of \( [\ell(\lambda)] \).
\end{prop}
\begin{proof}
  By Theorem~\ref{thm:skew_Omega} and the previous proposition, we have
  \begin{align*}
    e_k^\perp h_\lambda(\xx)
      &= \langle \ee_k(\uu) \Omega(\xx,\uu), \gamma_\lambda  \rangle \\
      &= \sum_{\mu} h_\mu(\xx) \langle \ee_k(\uu) \mono_\mu(\uu), \gamma_\lambda \rangle.
  \end{align*}
  Since \( \ee_k(\uu) \) is the sum of monomials whose indices are strictly
  decreasing, it suffices to find words in \( [\lambda] \) such that
  the first \( k \) letters are strictly decreasing and the remaining part are weakly
  increasing. These words can be obtained by choosing \( k \)-subsets of
  \( [\ell(\lambda)] \), and thus we have the desired identity.
\end{proof}

We next compute \( p_k^\perp h_\lambda(\xx) \) similarly.
\begin{prop} \label{prop:p_perp_h}
  For a partition \( \lambda \) and an integer \( k\ge 1 \),
  we have
  \[
    p_k^\perp h_\lambda(\xx)
    = \sum_{i=1}^{\ell(\lambda)} h_{\lambda-k \epsilon_i}(\xx).
  \]
\end{prop}
\begin{proof}
  Mimicking the relation \eqref{eq:relation_pe}, let
  \[
    \pp_k(\uu) := \ee_1(\uu) \hh_{k-1}(\uu) - 2\ee_2(\uu) \hh_{k-2}(\uu) + \cdots
                + (-1)^{k-1} k \ee_k(\uu) \in \UU.
  \]
  Or simply, let \( \pp_k(\uu) = \phi(p_k(\xx))\in\UU/\II_0 \).
  Then, thanks to Theorem~\ref{thm:skew_Omega}, it is enough to express
  the noncommutative analogue \( \pp_k(\uu) \) of the power sum symmetric function
  as a sum of monomials.
  As before, the ideal \( \II_0 \) makes it easy. Similar to the ordinary power
  sum symmetric function \( p_k(\xx) \), we can write \( \pp_k(\uu) \) as
  \[
    \pp_k(\uu) \equiv \sum_{i=1}^N u_i^k \mod \II_0.
  \]
  Then, the theorem completes the proof.
\end{proof}

\subsection{Schur functions}
Schur functions are the most significant basis for \( \Sym \),
with broad applications in combinatorics, representation theory,
algebraic geometry, and many other fields.
In this subsection, we derive the Littlewood--Richardson rule,
the rule for the structure constant of Schur functions,
using the theory of noncommutative symmetric functions.
It is worth noting that the contents of this section served as the starting
point for \cite{FG98}, the pioneering paper that first introduced the theory
of noncommutative symmetric functions.
In fact, some of the material in this section overlaps with \cite{FG98}.

We first recast basic properties of Schur functions.
While we had defined Schur functions as the determinants \eqref{eq:relation_se}
of the certain matrices consisting of elementary symmetric functions,
Schur functions can also be written combinatorially. For a partition \( \lambda \),
a \emph{semistandard Young tableaux} (abbreviated as SSYT) \( T \) of shape
\( \lambda \) is an array of positive integers of shape \( \lambda \) such that
each row is weakly increasing, and each column is strictly increasing.
The \emph{column word} \( \col(T) \) of \( T \) is the word obtained by reading
$T$ from bottom to top, beginning with the leftmost column of $T$ and working
from left to right. 
For example, \( T \) is a SSYT of shape \( (4,3,1) \):
\begin{center}
  \( T \) = \ytableaushort{1134,346,5}
\end{center}
The column word \( \col(T) \) is \( \mathsf{53141634} \).
Then the Schur function \( s_\lambda(\xx) \) is equal to a generating function for
semistandard Young tableaux of shape \( \lambda \)~(\cite[Corollary~7.16.2]{Sta99}):
\[
  s_\lambda(\xx) = \sum_{T\in\SSYT(\lambda)} x^T,
\]
where \( \SSYT(\lambda) \) is the set of semistandard Young tableaux of shape
\( \lambda \) and $x^T=\prod_i x_i^{m_i(T)}$ where $m_i(T)$ is the number of
appearances of $i$ in $T$.

One of significant properties of Schur functions is that
they form an orthonormal basis with respect to the Hall inner product.
Equivalently, the Cauchy product \( C(\xx,\yy) \) can be written as follows:
\begin{equation} \label{eq:C=ss}
  C(\xx,\yy) = \sum_{\lambda\in\Par} s_\lambda(\xx) s_\lambda(\yy).
\end{equation}

We now introduce a noncommutative analogue of Schur functions.
For a partition \( \lambda \), the \emph{noncommutative Schur function}
\( \sch_\lambda(\uu) \) is defined by
\[
  \sch_\lambda(\uu) = \sum_{\sigma\in\SYM_m} \sgn(\sigma)
    \ee_{\lambda'_1-\sigma_1+1}(\uu) \ee_{\lambda'_2-\sigma_2+2}(\uu)\cdots
    \ee_{\lambda'_m-\sigma_m+m}(\uu),
\]
where \( m=\lambda_1 \).
In addition, let \( \II_{\plac} \) be the 2-sided ideal of \( \UU \) generated by
\begin{align*}
  & u_a u_c u_b - u_c u_a u_b \quad \mbox{(\( 1\le a \le b < c \le N \))}, \\
  & u_b u_a u_c - u_b u_c u_a \quad \mbox{(\( 1\le a < b \le c \le N \))}.
\end{align*}
Note that \( \UU/\II_{\plac} \) is called the \emph{plactic algebra}~\cite{LS81}.
\begin{lem}[{\cite[Lemma~3.1]{FG98}}]
  The ideal \( \II_{\plac} \) satisfies the commutation relation.
\end{lem}
The plactic algebra has a close relationship with the RSK correspondence.
The RSK correspondence is a fundamental bijection between \( \NN \)-valued
matrices and pairs of semistandard Young tableaux of the same shape.
Specifically, it maps each word \(\ww\) to a pair of tableaux
\((P(\ww),Q(\ww))\), where \( P(\ww) \) is a SSTY whose content
coincides with that of the word.
While we do not describe the correspondence explicitly inhere, it is worth
noting that it gives a bijective proof of \eqref{eq:C=ss}.
The following are well-known.
\begin{lem}[{\cite[Lemma~A1.1.10]{Sta99}}] \label{lem:P(col(T))=T}
  For a semistandard Young tableau \( T \), we have \( P(\col(T)) = T \).
\end{lem}
\begin{thm}[{\cite[Theroem~A1.1.4]{Sta99}}] \label{thm:Knuth_equiv}
  For words \( \ww \) and \( \ww' \),
  \[
    u_\ww \equiv u_{\ww'} \mod \II_{\plac} \quad\mbox{if and only if}\quad
    P(\ww)=P(\ww').
  \]
\end{thm}
For a SSYT \( T \), let
\[
  \gamma_T = \sum_{P(\ww)=T} \ww.
\]
Then by Theorem~\ref{thm:Knuth_equiv}, \(\gamma_T\) belongs to \(\II_{\plac}^\perp\).

We now define a noncommutative version of Schur functions.
For a partition \( \lambda \), the \emph{noncommutative Schur function}
\( \sch_\lambda(\uu) \) is defined by
\[
  \sch_\lambda(\uu) =
      \sum_{\sigma\in\SYM_m} \sgn(\sigma)
          \ee_{\lambda'_1-\sigma_1+1}(\uu) \ee_{\lambda'_2-\sigma_2+2}(\uu)\cdots
          \ee_{\lambda'_m-\sigma_m+m}(\uu), 
\]
where \( m=\lambda_1 \); cf. \eqref{eq:relation_se}.
Surprisingly we can express the noncommutative Schur function as the ordinary
Schur function in \( \UU/\II_{\plac} \).
\begin{thm}[{\cite[Lemma~3.2]{FG98}}] \label{thm:sch=sum_T}
  For a partition \( \lambda \), we have
  \[
    \sch_\lambda(\uu) \equiv \sum_{T\in\SSYT(\lambda,N)} u_{\col(T)} \mod \II_{\plac},
  \]
  where \( \SSYT(\lambda, N) \) is the set of all semistandard Young tableaux
  with entries in \( [N] \).
\end{thm}
From Theorem~\ref{thm:Omega=gf}, Lemma~\ref{lem:P(col(T))=T},
Theorems~\ref{thm:Knuth_equiv} and
\ref{thm:sch=sum_T}, we deduce the following theorem.
\begin{thm}
  Let \( T \) be a semistandard Young tableau of shape \( \lambda \).
  Then we have
  \begin{equation} \label{eq:F_gamma=schur}
    F_{\gamma_T}(\xx) = s_{\lambda}(\xx).
  \end{equation}
\end{thm}

Let us now derive the Littlewood--Richardson rule.
We say a word \( \ww \) a \emph{tableau word of shape \( \lambda \)} if \( \ww \) is
the column word of some SSYT \( T \) of shape \( \lambda \).
For partitions \( \mu \) and \( \nu \), we write
\[
  s_\mu(\xx) s_\nu(\xx) = \sum_{\lambda} c^\lambda_{\mu\nu} s_\lambda(\xx)
\]
and call \( c^{\lambda}_{\mu\nu} \) the \emph{Littlewood--Richardson coefficient}.
Thanks to \eqref{eq:skewing_s=s_skew}, \eqref{eq:a_perp=sum_db} and
the duality of Schur functions~\eqref{eq:C=ss}, we have
\begin{equation} \label{eq:skew_s=LR}
  s_\mu^\perp s_\lambda(\xx)
    = s_{\lambda/\mu}(\xx) = \sum_{\nu} c^\lambda_{\mu\nu} s_\nu(\xx).
\end{equation}
We therefore present the Littlewood--Richardson rule, a combinatorial interpretation
for the Littlewood--Richardson coefficient, using Theorem~\ref{thm:skew_Omega}.
\begin{thm}
  Let \( \mu\vdash n,\nu\vdash m \) and \( \lambda\vdash n+m \), and 
  fix a semistandard Young tableau \( T \) of shape \( \lambda \).
  Then the Littlewood--Richardson coefficient \( c^\lambda_{\mu\nu} \) counts
  words \( \ww \) satisfying
  \begin{enumerate}[label=(\roman*)]
    \item \( P(\ww) = T \),
    \item \( \ww_1\cdots\ww_n \) is a tableau word of shape \( \mu \), and
    \item \( \ww_{n+1}\cdots\ww_{n+m} \) is a tableau word of shape \( \nu \).
  \end{enumerate}
\end{thm}
\begin{proof}
  By Theorem~\ref{thm:skew_Omega}, \eqref{eq:F_gamma=schur} and \eqref{eq:skew_s=LR},
  the coefficient can be computed as follows:
  \[
    c^\lambda_{\mu\nu} = \langle \sch_\mu(\uu) \sch_\nu(\uu), \gamma_T \rangle.
  \]
  Since \( \sch_\rho(\uu) \) is the generating function for tableau words of
  shape \( \rho \), and \( \gamma_T \) is the generating function for words \( \ww \)
  such that \( P(\ww)=T \), we have the proof.
\end{proof}

\subsection{Chromatic quasisymmetric functions} \label{sec:chrom_sym}
In \cite{Hwa22, BEPS22}, the authors settled a foundation of the theory
of noncommutative symmetric functions for posets.
Using the theory, they studied several expansions of chromatic quasisymmetric
functions of natural unit interval orders
(and more generally, \( \mathbf{(3+1)} \)-free posets).
In this section, we apply Theorem~\ref{thm:skew_Omega} to their setting,
and provide two recurrence relations \eqref{eq:recur_e_perp} and
\eqref{eq:recur_p_perp} between chromatic quasisymmetric functions.
One of the relations is a generalization of Harada--Precup's recurrence relation.
These two relations do not give
the positivity of the coefficients \( c^{P,\beta}_\lambda(q) \),
but we hope that they provide intuition for finding a combinatorial
interpretation of the coefficients.

We first review the definition of chromatic quasisymmetric functions and
the noncommutative symmetric function theory for posets.
Throughout this section, we stand \( P \) for a poset on \( [N] \), and denote
\( >_P \) the partial order.
For \( a,b\in [N] \), we also write \( a\sim_P b \)
if \( a=b \) or they are incomparable in \( P \).
For a word \( \ww \) of length \( d \) on alphabet \( [N] \), define
the \emph{\( P \)-descent set} of \( \ww \)
\[
  \Des_P(\ww) = \{i\in[d-1] \mid \ww_i >_P \ww_{i+1}  \},
\]
and the \emph{\( P \)-inversion number} of \( \ww \)
\[
  \inv_P(\ww) = | \{(i,j)\in [d]\times[d] \mid \mbox{\( i<j \), \( \ww_i>\ww_j \), and \( \ww_i\sim_P \ww_j \)} \} |.
\]
For \( \beta\in\NN^N \), let \( W(\beta) \) be the set of all words of
content \( \beta \) on alphabet \( [N] \), i.e., each letter \( i \) occurs
\( \beta_i \) times in the word for all \( i\in [N] \).
The \emph{chromatic quasisymmetric function} \( \omega X_P(\xx;q,\beta) \)
of \( P \) of type \( \beta \) is defined by
\[
  \omega X_P(\xx;q,\beta) = \sum_{\ww\in W(\beta)} q^{\inv_P(\ww)} F_{d,\Des_P(\ww)}(\xx).
\]

For \( k\ge 1 \), the \emph{noncommutative \( P \)-elementary symmetric function}
\( \ee_k^P(\uu) \) is defined by
\[
  \ee_k^P(\uu) = \sum_{i_1>_P\cdots >_P i_k} u_{i_1}\cdots u_{i_k} \in \UU.
\]
We also define \( P \)-analogues of the noncommutative symmetric functions
\( \hh_\lambda(\uu), \pp_k(\uu) \) and \( \mono_\lambda(\uu) \) in the same way
using \( \ee_k^P(\uu) \) instead of \( \ee_k(\uu) \).
Similar to \eqref{eq:hh=increasing}, we then easily obtain
\[
  \hh_k^P(\uu) = \sum_{i_1 \ngtr_P \cdots \ngtr_P i_k} u_{i_1}\cdots u_{i_k}.
\]
The \emph{noncommutative \( P \)-Cauchy product} is
\[
  \Omega^P(\xx,\uu) = H^P(x_1,\uu) H^P(x_2,\uu) \cdots \in \UU[[\xx]],
\]
where \( H^P(x,\uu) = \sum_{\ell\ge 0} x^\ell \hh^P_\ell(\uu) \).
We can therefore write \( \Omega^P(\xx,\uu) \) as follows:
\[
  \Omega^P(\xx,\uu) = \sum_\ww u_\ww F_{d, \Des_P(\ww)}(\xx),
\]
where the sum ranges over all words on \( [N] \), and \( d \) is the length of
\( \ww \).
For \( \gamma\in\UU^* \) (or \( \gamma\in\ZZ[q]\otimes \UU^* \)),
write
\[
  F^P_\gamma(\xx) = \langle \Omega^P(\xx,\uu), \gamma \rangle.
\]
Therefore, letting
\[
  \gamma_\beta = \sum_{\ww\in W(\beta)} q^{\inv_P(\ww)} \ww,
\]
we have
\[
  \omega X_P(\xx;q,\beta) = F^P_{\gamma_\beta}(\xx).
\]
We say that a 2-sided ideal \( \II \) of \( \UU \) \emph{satisfies the \( P \)-commutation relation} if for any \( k,\ell \ge 1 \),
\[
  \ee_k^P(\uu) \ee_\ell^P(\uu) \equiv \ee_\ell^P(\uu) \ee_k^P(\uu) \mod \II.
\]
Let \( \phi^P \) be a mapping
\[
  e_k(\xx)\in\Sym \mapsto \ee_k^P(\uu)\in\UU
\]
for all \( k \ge 1 \).
Then a very similar way to the proof of Theorem~\ref{thm:skew_Omega}
gives the following \( P \)-analogue of the theorem.
\begin{thm} \label{thm:skew_Omega_P}
  Choose a symmetric function \( f(\xx)\in\Sym \).
  Let \( \II \) be a 2-sided ideal satisfying the \( P \)-commutation relation,
  and \( \ff^P(\uu) := \phi^P(f(\xx)) \).
  Let \( \gamma\in\II^\perp \), then we have
  \[
    f^\perp F^P_\gamma(\xx) = \langle \ff^P(\uu) \Omega^P(\xx,\uu), \gamma \rangle.
  \]
\end{thm}

We now focus on a certain class of posets.
A \emph{natural unit interval order} \( P \) on \( [N] \) is a poset satisfying that
\begin{enumerate}
  \item if \( a<_P b \), then \( a < b \);
  \item if \( a <_P c \), \( a\sim_P b \), and \( b\sim_P c \), then \( a<b<c \).
\end{enumerate}
In general, \( \omega X_P(\xx;q,\beta) \) is quasisymmetric,
but Shareshian and Wachs~\cite{SW16} showed that the chromatic quasisymmetric
function of a natural unit interval order is symmetric.
\begin{rmk}
  Natural unit interval orders have another characterization:
  a poset \( P \) is a natural unit interval order if and only if
  \( P \) is \(\mathbf{(3+1)}\)-free and \( \mathbf{(2+2)} \)-free~\cite{SS58}.
  In \cite{Sta95}, introducing chromatic symmetric functions,
  Stanley concentrated his focus on \( \mathbf{(3+1)} \)-free posets.
  But, unfortunately, there is no `natural' way to label elements of the posets,
  and thus one cannot define the chromatic quasisymmetric functions of
  \( \mathsf{(3+1)} \)-free posets canonically.
  Therefore, Shareshian and Wachs~\cite{SW16} considered only natural unit
  interval orders. We will also deal with natural unit interval orders because
  we would like to state our recurrences with \( q \)-weights.
  Nevertheless, we highlight that all definitions and results appearing in what
  follows can be extended to \( \mathbf{(3+1)}\)-free posets with \( q=1 \).
\end{rmk}

For a natural unit interval order \( P \),
we define \( \II_P \) to be the 2-sided ideal of \( \UU \) generated by
\begin{align}
  & u_a u_c - u_c u_a && \mbox{(\( a<_P c  \))}, \label{rel:ideal_P_1} \\
  & u_b u_a u_c - u_a u_c u_b && \mbox{(\( a\sim_P b, b\sim_P c\) and \( a<_P c \))}.
  \label{rel:ideal_P_2}
\end{align}
Then we have an important proposition.
\begin{prop}[{\cite[Theorem~4.22]{Hwa22},~\cite[Proposition~3.24]{BEPS22}}]
  For a natural unit interval order \( P \),
  the ideal \( \II_P \) satisfies the \( P \)-commutation relation.
\end{prop}

For \( \beta\in\NN^N \), we write
\begin{equation} \label{eq:XP=ch}
  \omega X_P(\xx;q,\beta) = \sum_{\lambda} c_\lambda^{P,\beta}(q) h_\lambda(\xx).
\end{equation}
The positivity of \( c^{P,\beta}_\lambda(q) \) is one of the most famous conjecture
in algebraic combinatorics.
We present recurrences for the coefficients \( c^{P,\beta}_\lambda(q) \) using skewing
operators.
\begin{rmk}
  In \cite{Hwa22}, the author introduced an equivalence relation \( \sim \) on 
  words in \( W(\beta) \), and studied the symmetric function associated to
  each equivalence class (called a \( H \)-graph in \cite{BEPS22}),
  instead of \( \omega X_P(\xx;q,\beta) \).
  More precisely, for an equivalence class \( [\ww]\in \equivclass{W(\beta)} \),
  let
  \[
    \gamma_{[\ww]} = \sum_{\ww'\in[\ww]} \ww' \qand
    F^P_{\gamma_{[\ww]}}(\xx) = \langle \Omega^P(\xx,\uu), \gamma_{[\ww]} \rangle,
  \]
  and he showed several positivities of \( F^P_{\gamma_{[\ww]}}(\xx) \).
  By definition, the symmetric functions \( F^P_{\gamma_{\ww}}(\xx) \) can be
  thought of as refinements of \( \omega X_P(\xx;q,\beta) \), and  the positivity of
  \(F^P_{\gamma_{[\ww]}}(\xx) \) implies the positivity of \( \omega X_P(\xx;q,\beta) \).
  In fact, the generator \eqref{rel:ideal_P_2} of the ideal \( \II_P \) is just an
  algebraic expression of the equivalence relation. While the recurrences we will
  present hold not only for the chromatic quasisymmetric functions,
  but also for the refinements,
  we only state them at the level of the chromatic quasisymmetric function.
  We believe the the reader can easily reformulate the recurrences
  for \( F^P_{\gamma_{[\ww]}}(\xx) \).
\end{rmk}

We first consider the skewing operator \( e_k^\perp \).
For \( \beta\in\NN^N \) and \( i\in [N] \), define
\[
  \deg_P(i,\beta) = \sum_{k\in [N]} \beta_i \chi(\mbox{\( i>k \) and \( i\sim_P k \)}),
\]
where \( \chi(Q) \) is the characteristic map, that is, 1 if property \( Q \) is
true, and 0 otherwise. Moreover, for a multiset \( S \) of letters in \( [N] \),
define
\[
  \deg_P(S, \beta) = \sum_{i\in S} \deg_P(i, \beta).
\]
In addition, let
\[
  \beta_S = ({\beta_S}_1,\dots, {\beta_S}_N)\in\NN^N
\]
such that \( {\beta_S}_i \) equals the number of appearances of \( i \) in \( S \).
For \( k\ge 1 \), we denote by \( C_k(P) \) the set of chains of length \( k \)
in \( P \). By convention, if \( \beta \) contains a negative integer entry, 
then we let \( \omega X_P(\xx;q,\beta) = 0 \).
\begin{thm} \label{thm:recur_e_perp}
  Let \( P \) be a natural unit interval order, \( \beta\in\NN^N \), and
  \( d=\beta_1+\dots+\beta_N \).
  For an integer \( k\ge 1 \) and a partition \( \lambda\vdash d-k \), we have
  \begin{equation} \label{eq:recur_e_perp}
    \sum_{(\mu, S)} c^{P,\beta}_{\mu}(q)
    = \sum_{T\in C_k(P)} q^{\deg_P(T,\beta-\beta_T)} c^{P,\beta-\beta_T}_\lambda(q),
  \end{equation}
  where \( (\mu, S) \) is a pair of a partition \( \mu\vdash d \) and
  a \( k \)-subset \( S \) of \( [\ell(\mu)] \) such that \( \mu-\epsilon_S \)
  equals \( \lambda \), after reordering in weakly decreasing order.
\end{thm}
\begin{proof}
  We compute \( e_k^\perp \omega X_P(\xx;q,\beta) \) in two different ways.
  On the one hand, by applying the skewing operator \( e_k^\perp \) to
  \eqref{eq:XP=ch} and Proposition~\ref{prop:e_perp_h}, we have
  \begin{align}
    e_k^\perp \omega X_P(\xx;q,\beta)
      &= \sum_{\mu\vdash d} c^{P,\beta}_\mu(q) e_k^\perp h_\mu(\xx) \nonumber \\
      &= \sum_{\mu\vdash d} c^{P,\beta}_\mu(q) \sum_{S\in\binom{[\ell]}{k}} h_{\mu-\epsilon_S}(\xx), \label{eq:e_perp_XP_1}
  \end{align}
  where \( \ell \) is the length of \( \mu \).
  On the other hand, Theorem~\ref{thm:skew_Omega_P} yields
  \[
    e_k^\perp \omega X_P(\xx;q,\beta) = \langle \ee_k^P(\uu) \Omega^P(\xx,\uu), \gamma_{\beta} \rangle.
  \]
  Since \( e_k^P(\uu) \) is the generating function for chains of \( P \) of length
  \( k \), if \( \ww\in W(\beta) \) contributes to \eqref{eq:e_perp_XP_2},
  then the \( k \)-prefix of \( \ww \) should be a \( P \)-decreasing word,
  i.e., \( \ww_1>_P \cdots >_P \ww_k \).
  In addition, the \( P \)-inversion number of such words can be computed as
  \[
    \inv_P(\ww) = \deg_P(T, \beta-\beta_T) + \inv_P(\ww_{k+1}\cdots\ww_d),
  \]
  where \( T=\{\ww_1,\dots,\ww_k\} \), which forms a chain in \( P \). Then we have
  \begin{align}
    \langle \ee_k^P(\uu) \Omega^P(\xx,\uu), \gamma_{\beta} \rangle
      &= \sum_{T\in C_k(P)} q^{\deg_P(T,\beta-\beta_T)} \langle \Omega^P(\xx,\uu), \gamma_{\beta-\beta_T} \rangle \nonumber \\
      &= \sum_{T\in C_k(P)} q^{\deg_P(T,\beta-\beta_T)} \omega X_P(\xx;q,\beta-\beta_T) \nonumber \\
      &= \sum_{\nu\vdash d-k} \sum_{T\in C_k(P)} q^{\deg_P(T,\beta-\beta_T)} c^{P,\beta-\beta_T}_\nu(q) h_\nu(\xx). \label{eq:e_perp_XP_2}
  \end{align}
  Therefore, equating \eqref{eq:e_perp_XP_1} and \eqref{eq:e_perp_XP_2}, and
  comparing the coefficients give the proof.
\end{proof}
\begin{exam}
  Let \( P \) be the poset on \( [5] \) such that for \( i,j\in P \), \( i<_P j \)
  if and only if \( j-i \ge 2 \). It is easy to see that \( P \) is a natural
  unit interval order. Let \( \beta=(1,1,2,1,1) \), and \( k=2 \).
  Then the chromatic quasisymmetric function is given by
  \begin{align*}
    \omega X_P(\xx;q,\beta) &= q^{3}h_{3,2,1} + \left(q^{4} + q^{3} + q^{2}\right)h_{3,3} + \left(q^{4} + q^{3} + q^{2}\right)h_{4,1,1} \\
    &\qquad + \left(q^{4} + 2 q^{3} + q^{2}\right)h_{4,2} + \left(2 q^{5} + 3 q^{4} + 3 q^{3} + 3 q^{2} + 2 q\right)h_{5,1}  \\
    &\qquad + \left(q^{6} + 2 q^{5} + 2 q^{4} + 2 q^{3} + 2 q^{2} + 2 q + 1\right)h_{6}.
  \end{align*}
  Choose \( \lambda=(3,1) \), and let us verify Theorem~\ref{thm:recur_e_perp}.
  To compute the left-hand side of \eqref{eq:recur_e_perp}, we need to find
  appropriate pairs \( (\mu, S) \): \(((4,2), \{1,2\})\), \(((4,1,1), \{1,2\})\), 
  \(((4,1,1), \{1,3\})\), \(((3,2,1), \{2,3\})\), \( ((2,1,1,1),\{2,3\}) \),
  \( ((2,1,1,1),\{2,4\}) \) and \( ((2,1,1,1),\{3,4\}) \).
  Then we have
  \begin{align*}
    \mbox{(LHS)}
      &= c^{P,\beta}_{42}(q) + 2 c^{P,\beta}_{411}(q) + c^{P,\beta}_{321}(q) + 
          3 c^{P,\beta}_{2111}(q) \\
      &= 3q^4 + 5q^3 + 3q^2.
  \end{align*}
  For computing the right-hand side, we first find chains of length 2.
  In our case, by definition, \( \{i,j\} \) forms a chain if and only if \( j-i\ge 2 \).
  For each chain \( T\in C_2(P) \), we can compute \( \deg_P(T, \beta-\beta_T) \)
  and \( c^{P, \beta-\beta_T}_{31}(q) \) as follows:
  \[
    \begin{tabular}{|c|c|c|}
      \hline  
      \(T\) & \(\deg_P(T,\beta-\beta_T)\) & \(c^{P,\beta-\beta_T}_{31}(q)\) \\
      \hline\hline
      \( 1,3 \) & \( 2 \) & \( q^2 + q \) \\
      \hline
      \( 1,4 \) & \( 2 \) & \( q^2 + q + 1 \) \\
      \hline
      \( 1,5 \) & \( 1 \) & \( q^2 \) \\
      \hline
      \( 2,4 \) & \( 3 \) & \( 0 \) \\
      \hline
      \( 2,5 \) & \( 2 \) & \( q^2 + q + 1 \) \\
      \hline
      \( 3,5 \) & \( 1 \) & \( q^2 + q \) \\
      \hline
    \end{tabular}
  \]
  Then we obtain
  \begin{align*}
    \mbox{(RHS)}
      &= \sum_{T\in C_2(P)} q^{\deg_P(T, \beta-\beta_T)} c^{P,\beta-\beta_T}_{31}(q) \\
      &= 3q^4 + 5q^3 + 3q^2,
  \end{align*}
  which justifies the theorem.
\end{exam}
Once one choose \( k \) and \( \lambda \) carefully in the theorem,
the left-hand side of \eqref{eq:recur_e_perp} could consist of a single term.
Such cases coincide with Harada--Precup's recurrence relation.
To be precise, let \( h_P \) be the length of a longest chain in \( P \).
(This value is called the height, or the bounce number of \( P \).)
Then, by definition, \( \ee_k^P(\uu)=0 \) for \( k > h_P \). Using this,
one can easily see that \( c^{P, \beta}_\lambda(q)=0 \) for partitions
\( \lambda \) with \( \ell(\lambda)> h_P \).
We now choose a partition \( \mu \) whose length is \( h_P \). Let \( k=h_P \),
and \( \lambda=(\lambda_1,\dots,\lambda_k) \) such that \( \lambda_i = \mu_i-1 \)
for each \( i \).
In this case, \eqref{eq:recur_e_perp} becomes
\begin{equation} \label{eq:Harada--Precup}
  c^{P,\beta}_\mu(q)
    = \sum_{T\in C_k(P)} q^{\deg_P(T,\beta-\beta_T)} c^{P,\beta-\beta_T}_\lambda(q),
\end{equation}
which is the Harada--Precup recurrence.
In \cite{HP19}, Harada and Precup showed the recurrence~\eqref{eq:Harada--Precup}
for the case \( h_P=2 \) using geometry of Hessenberg varieties,
and conjectured for any natural unit interval order.
In \cite{Hwa22}, the author proved the conjecture using the noncommutative symmetric
function theory.

Next, let us skew the chromatic quasisymmetric function by the power sum symmetric
function \( p_k(\xx) \).
Before applying Theorem~\ref{thm:skew_Omega_P}, we need to define a \( P \)-analogue
of the noncommutative power sum symmetric function \( \pp_k(\uu) \), and find a
monomial expression for it. As before, for \( k\ge 1 \), define
\[
  \pp_k^P(\uu) = \ee^P_1(\uu) \hh^P_{k-1}(\uu) - 2\ee^P_2(\uu) \hh^P_{k-2}(\uu)
                + \cdots + (-1)^{k-1} k \ee^P_k(\uu) \in \UU,
\]
so that \( \pp_k^P(\uu) = \phi^P(p_k(\xx)) \) in \( \UU/\II_P \).
To give a monomial expression of \( \pp_k(\uu) \), we need additional notions.
For a word \( \ww=\ww_1\cdots\ww_d \) on \( [N] \),
we say that \emph{\( \ww \) has no \( P \)-descents}
if for each \(1 \le i \le d-1 \), \( \ww_i \ngtr_P \ww_{i+1} \).
We also say that \emph{$\ww$ has no nontrivial
left-to-right $P$-maxima} if for each $2\le i\le d$,
there is an integer $j<i$ such that $\ww_i \ngtr_P \ww_j$.
Let \( \calN_{P, k} \) be the set of words \( \ww \) of length \( k \) on \( [N] \)
such that \( \ww \) has no \( P \)-descents and no nontrivial left-to-right
\( P \)-maxima.
\begin{prop}[{\cite{SW16,Hwa22}}] \label{prop:pp_k=N}
  For \( k\ge 1 \), we have
  \[
    \pp^P_k(\uu) \equiv \sum_{\ww\in\calN_{P,k}} u_\ww \mod \II_P.
  \]
\end{prop}
Note that the proposition holds for any poset with an ideal containing
\eqref{rel:ideal_P_1} in place of \( \II_P \); see \cite[Remark~4.13]{Hwa22}.
We also remark that \( \pp^P_k(\uu)=\phi^P(p_k(\xx)) = \phi^P(m_k(\xx)) \)
so that we have
\begin{equation} \label{eq:pp_k=c_k}
  \langle \pp^P_k(\uu), \gamma_{\beta} \rangle = c^{P, \beta}_k(q),
\end{equation}
the coefficient of \( h_k(\xx) \) of \( \omega X_P(\xx;q,\beta) \),
because the two bases \( \{m_\lambda(\xx)\}_{\lambda\in\Par} \)
and \( \{h_\lambda(\xx)\}_{\lambda\in\Par} \) are dual to each other.

For \( \alpha\in\NN^N \), define the multiset
\( S_\alpha=\{1^{\alpha_1},\dots,N^{\alpha_N}\} \).
For \( \alpha,\beta \in \NN^N \), we write \( \alpha \le \beta \) if
\( \alpha_i \le \beta_i \) for all \( i\in[N] \).
\begin{thm} \label{thm:recur_p_perp}
  Let \( P \) be a natural unit interval order, \( \beta\in\NN^N \), and
  \( d=\beta_1+\dots+\beta_N \).
  For an integer \( k\ge 1 \) and a partition \( \lambda\vdash d-k \), we have
  \begin{equation} \label{eq:recur_p_perp}
    \sum_{(\mu, i)} c^{P,\beta}_{\mu}(q)
    = \sum_{\substack{\alpha\le \beta \\ \alpha_i+\dots+\alpha_N=k}}
    q^{\deg_P(S_\alpha, \beta-\alpha)} c^{P,\alpha}_k(q,\alpha) c^{P,\beta-\alpha}_\lambda(q),
  \end{equation}
  where \( (\mu, i) \) is a pair of a partition \( \mu\vdash d \) and
  \( 1\le i\le \ell(\mu) \) such that \( \mu_i\ge k \) and \( \mu-k \epsilon_i \)
  equals \( \lambda \), after reordering in weakly decreasing order.
\end{thm}
\begin{proof}
  Similar to the proof of Theorem~\ref{thm:recur_e_perp}, we will establish the
  theorem by evaluating \( p_k^\perp \omega X_P(\xx;q,\beta) \) through two
  different ways, and then equating the results.

  First, applying the skewing operator \( p_k^\perp \) to \eqref{eq:XP=ch},
  we have
  \begin{align*}
    p_k^\perp \omega X_P(\xx;q,\beta)
      &= \sum_\mu c^{P,\beta}_\mu p_k^\perp h_\mu(\xx) \\
      &= \sum_\mu c^{P,\beta}_\mu \sum_{i=1}^{\ell(\mu)} h_{\mu-k \epsilon_i}(\xx),
  \end{align*}
  by Proposition~\ref{prop:p_perp_h}.
  On the other hand, splitting each word \( \ww\in W(\beta) \) into
  \( \ww_1\cdots\ww_k \) and \( \ww_{k+1}\cdots\ww_d \),
  we can identify \( W(\beta) \) with
  \[
    \bigsqcup_{\substack{\alpha\le \beta \\ \alpha_1+\dots+\alpha_N=k}}
      W(\alpha) \times W(\beta-\alpha).
  \]
  In addition, for \( \ww\in W(\alpha)\times W(\beta-\alpha) \), one can easily see
  \[
    \inv_P(\ww) = \inv_P(\ww_1\cdots\ww_k) + \inv_P(\ww_{k+1}\cdots\ww_d)
      + \deg_P(S_\alpha, \beta-\alpha).
  \]
  Therefore we deduce
  \begin{align*}
    p_k^\perp \omega X_P(\xx;q,\beta)
      &= \langle \pp^P_k(\uu) \Omega^P(\xx,\uu), \gamma_\beta \rangle && \mbox{by Theorem~\ref{thm:skew_Omega_P}}, \\
      &= \sum_{\substack{\alpha\le \beta \\ \alpha_1+\dots+\alpha_N=k}}
          q^{\deg_P(S_\alpha, \beta-\alpha)}
          \langle \pp^P_k(\uu), \gamma_\alpha \rangle ~
          \langle \Omega^P(\xx,\uu), \gamma_{\beta-\alpha} \rangle \\
      &= \sum_{\substack{\alpha\le \beta \\ \alpha_1+\dots+\alpha_N=k}}
      q^{\deg_P(S_\alpha, \beta-\alpha)} c^{P,\alpha}_k(q)
      \langle \Omega^P(\xx,\uu), \gamma_{\beta-\alpha} \rangle && \mbox{by \eqref{eq:pp_k=c_k},}  \\
      &= \sum_{\substack{\alpha\le \beta \\ \alpha_1+\dots+\alpha_N=k}}
        \sum_{\nu\vdash d-k}
        q^{\deg_P(S_\alpha, \beta-\alpha)} c^{P,\alpha}_k(q)
        c^{P,\beta-\alpha}_\nu(q) h_\nu(\xx).
  \end{align*}
  Equating the two results completes the proof.
\end{proof}
\begin{exam}
  Continuing the previous example, let \( P \) and \( \beta \) be the ones
  in the example. We also let \( \lambda=(3,1) \) and \( k=2 \).
  The following is a list of pairs \( (\mu,i) \) satisfying the conditions in the
  theorem:
  \(((5,1), 1)\), \(((3,3), 1)\), \(((3,3), 2)\), and \(((3,2,1), 2)\).
  We then have
  \begin{align*}
    \mbox{(LHS)} 
      &= c^{P,\beta}_{51}(q) + 2c^{P,\beta}_{33}(q) + c^{P,\beta}_{321}(q) \\
      &= 2 q^{5} + 5 q^{4} + 6 q^{3} + 5 q^{2} + 2 q.
  \end{align*}
  On the other hand, there are many \( \alpha\le \beta \) with
  \( \alpha_1+\dots+\alpha_5=2 \), and for some such \( \alpha \),
  the value \( c^{P,\alpha}_2(q,\alpha) \) is equal to \( 0 \). Thus we only present
  \( \alpha \) with nonzero \( c^{P,\alpha}_2(q,\alpha) \):
  \[
    \begin{tabular}{|c|c|c|c|}
      \hline  
      \(\alpha\) & \( \deg_P(S_\alpha,\beta-\alpha) \) & \( c^{P,\alpha}_2(q,\alpha) \) & \(c^{P,\beta-\alpha}_{31}(q)\) \\
      \hline\hline
      \( (1,1,0,0,0) \) & 2 & \( q+1 \) & \( q^2 + q \) \\
      \hline
      \( (0,1,1,0,0) \) & 2 & \( q+1 \) & \( q^2 + q + 1 \) \\
      \hline
      \( (0,0,2,0,0) \) & 2 & \( 1 \) & \( 0 \) \\
      \hline
      \( (0,0,1,1,0) \) & 1 & \( q+1 \) & \( q^2 + q + 1 \) \\
      \hline
      \( (0,0,0,1,1) \) & 0 & \( q+1 \) & \( q^2 + q \) \\
      \hline
    \end{tabular}
  \]
  Therefore, we can easily verify
  \[
    \mbox{(RHS)} =  2 q^{5} + 5 q^{4} + 6 q^{3} + 5 q^{2} + 2 q.
  \]
\end{exam}


\bibliographystyle{alpha}
\bibliography{skewing}

\begin{thebibliography}{BEPS22}

\bibitem[BEPS22]{BEPS22}
Jonah Blasiak, Holden Eriksson, Pavlo Pylyavskyy, and Isaiah Siegl.
\newblock Noncommutative schur functions for posets.
\newblock {\em arXiv preprint arXiv:2211.03967}, 2022.

\bibitem[BF17]{BF17}
Jonah Blasiak and Sergey Fomin.
\newblock {Noncommutative Schur functions, switchboards, and Schur positivity}.
\newblock {\em Selecta Mathematica}, 23(1):727--766, 2017.

\bibitem[FG98]{FG98}
Sergey Fomin and Curtis Greene.
\newblock {Noncommutative Schur functions and their applications}.
\newblock {\em Discrete Mathematics}, 193(1-3):179--200, 1998.

\bibitem[HP19]{HP19}
Megumi Harada and Martha~E Precup.
\newblock {The cohomology of abelian Hessenberg varieties and the
  Stanley--Stembridge conjecture}.
\newblock {\em Algebraic Combinatorics}, 2(6):1059--1108, 2019.

\bibitem[Hwa22]{Hwa22}
Byung-Hak Hwang.
\newblock {Chromatic quasisymmetric functions and noncommutative $P$-symmetric
  functions}.
\newblock {\em arXiv preprint arXiv:2208.09857}, 2022.

\bibitem[LS81]{LS81}
Alain Lascoux and Marcel-P Sch{\"u}tzenberger.
\newblock {Le mono{\i}de plaxique}.
\newblock {\em Noncommutative structures in algebra and geometric combinatorics
  (Naples, 1978)}, 109:129--156, 1981.

\bibitem[SS58]{SS58}
Dana Scott and Patrick Suppes.
\newblock {Foundational aspects of theories of measurement 1}.
\newblock {\em The journal of symbolic logic}, 23(2):113--128, 1958.

\bibitem[Sta95]{Sta95}
Richard~P Stanley.
\newblock {A symmetric function generalization of the chromatic polynomial of a
  graph}.
\newblock {\em Advances in Mathematics}, 111(1):166--194, 1995.

\bibitem[Sta99]{Sta99}
Richard~P Stanley.
\newblock Enumerative combinatorics. vol. 2, volume 62 of.
\newblock {\em Cambridge Studies in Advanced Mathematics}, 1999.

\bibitem[SW16]{SW16}
John Shareshian and Michelle~L Wachs.
\newblock {Chromatic quasisymmetric functions}.
\newblock {\em Advances in Mathematics}, 295:497--551, 2016.

\end{thebibliography}

\end{document}